\documentclass[11pt,a4paper]{amsart}

\usepackage{fancyhdr,euler,stmaryrd} 
\usepackage{amssymb} 
\usepackage{combelow} 
\usepackage{tensor} 
\usepackage{hyperref} 
\usepackage[all]{xy} 
\usepackage{mathrsfs} 
\usepackage{indentfirst} 
\usepackage{mathtools} 
\usepackage{titlesec} 
\usepackage[a4paper]{geometry}

\newtheorem{theorem}{Theorem}[section]
\newtheorem{corollary}[theorem]{Corollary}
\newtheorem{lemma}[theorem]{Lemma}
\newtheorem{proposition}[theorem]{Proposition}
\theoremstyle{definition}
\newtheorem{definition}[theorem]{Definition}
\newtheorem{remark}[theorem]{Remark}
\newtheorem{subsec}[theorem]{}
\renewenvironment{proof}{\par\noindent\textbf{Proof:}}{\hfill $\blacksquare$\par\smallskip}

\renewcommand{\b}{\mathrm{b}}
\newcommand{\C}{\mathcal{C}}
\newcommand{\D}{\mathrm{\mathcal{D}}}
\newcommand{\G}{\bar{G}}
\renewcommand{\H}{\mathrm{\mathcal{H}}}
\newcommand{\K}{\mathcal{K}}
\DeclareFontFamily{OT1}{pzc}{}
\DeclareFontShape{OT1}{pzc}{m}{it}{<->s*[1.10] pzcmi7t}{}
\DeclareMathAlphabet{\mathscr}{OT1}{pzc}{m}{it}
\renewcommand{\k}{\mathscr{k}}
\newcommand{\lotimes}{\stackrel{\mathbf{L}}{\otimes}}
\renewcommand{\O}{\mathcal{O}}
\newcommand{\op}{\mathrm{op}}

\newcommand{\Db}[1]{\D^{\b}(#1)}
\newcommand{\Hb}[1]{\H^{\b}(#1)}
\newcommand{\Hom}[3]{\mathrm{Hom}_{#1}(#2,#3)}
\newcommand{\RHom}[3]{\mathrm{\mathbf{R}Hom}_{#1}(#2,#3)}
\newcommand{\para}[1]{\left(#1\right)}
\newcommand{\set}[1]{\left\{#1\right\}}
\newcommand{\triple}[3]{\left(#1,#2,#3\right)}

\newcommand{\titlename}	
{Blockwise relations between triples, and derived equivalences for wreath products}
\newcommand{\shorttitlename}
{Derived equivalences for wreath products}

\newcommand{\authorname}      {Andrei Marcus$^1$ \and Virgilius-Aurelian Minu\cb{t}\u{a}$^2$}
\newcommand{\pdfauthorname}   {Andrei Marcus and Virgilius-Aurelian Minuta}
\newcommand{\shortauthorname} {A. Marcus \and V. A. Minu\cb{t}\u{a}}
\newcommand{\universityname}  {$^{1,2}$Babe\cb{s}-Bolyai University of Cluj-Napoca, Romania}
\newcommand{\facultyname}     {Faculty of Mathematics and Computer Science}
\newcommand{\departmentname}  {Department of Mathematics}
\newcommand{\pmailaddress}        {marcus@math.ubbcluj.ro}
\newcommand{\smailaddress}        {minuta.aurelian@math.ubbcluj.ro}

\newcommand{\articleabstract}{Motivated by the reduction techniques involving character triples for the local-global conjectures, we show that a blockwise relation between module triples is a consequence of a derived equivalence with additional properties. Moreover, we show that this relation is compatible with wreath products.}

\newcommand{\msc}{20C20, 16W50,  20C05, 20E22, 16D90, 16E35}

\newcommand{\keywordterms}{Group algebra, block, Brauer map, graded algebra,  wreath product, derived equivalence}

\hypersetup{colorlinks = true, linkcolor = blue, anchorcolor =red, citecolor = blue, filecolor = red, urlcolor = blue, pdfauthor=\pdfauthorname, pdftitle=\titlename, pdfkeywords=\keywordterms}
\titleformat{\section}{\Large\bfseries}{\thesection}{1em}{}
\fancypagestyle{mypagestyle}{
  \fancyhf{}
  \fancyhead[OC]{\textit{\shorttitlename}}
  \fancyhead[EC]{\textit{\shortauthorname}}
  \fancyfoot[C]{\medskip $\leftslice\,$\thepage$\,\rightslice$ }
  
}
\fancypagestyle{firstpagestyle}{
  \fancyhf{}
  \fancyfoot[C]{\medskip $\leftslice\,$\thepage$\,\rightslice$ }
  
}
\title[\shorttitlename]{\LARGE{\titlename}}
\newcommand{\institution}{
\universityname\\
\facultyname\\
\departmentname}
\author[\shortauthorname]{\Large{\authorname}
\medskip\\
{\footnotesize \institution\\
$\begin{array}{l}
\text{email}^1\text{: \texttt{\href{mailto:\pmailaddress}{\pmailaddress}}}\\
\text{email}^2\text{: \texttt{\href{mailto:\smailaddress}{\smailaddress}}}
\end{array}$
}}
\begin{document}
\begin{abstract}
\articleabstract\\[0.1cm]
\textsc{MSC 2010.} \msc\\[0.1cm]
\end{abstract}
\begingroup
\def\uppercasenonmath#1{} 
\let\MakeUppercase\relax 
\maketitle
\endgroup
\thispagestyle{firstpagestyle}

\section{Introduction}

One of the main approaches to the local-global conjectures in the modular representation theory of finite groups is to show that they are a consequence of some inductive conditions on simple groups. Recent results (see Britta Sp\"ath's surveys \cite{ch:Spath2017} and \cite{ch:Spath2018}) use the language of character triples and of the various relations between them (denoted by $\ge_c$ and $\ge_b$), in order to obtain these reduction theorems.

This paper comes, first, as a followup to the study done in \cite{art:MM2}, in which we have given a  version for module triples of the relation $\ge_c$, and we have proved in \cite[Theorem 6.7]{art:MM2} that it is a consequence of a group graded Rickard equivalence with additional properties. Here, in Definition  \ref{d:rel-b} we also provide a module triple version of the relation $\ge_b$ (see \cite[Definition 4.2]{ch:Spath2018}), and we prove in Proposition \ref{prop:block_order} that  this too is a consequence of a special type of group graded derived equivalences which is compatible in  a certain sense with the Brauer map.

Our second objective is to build group graded derived  and  Rickard equivalences for  wreath products. Some technical details are already developed in \cite{art:MVA2} for Morita equivalences. Such constructions are again motivated by the reduction methods, which require the compatibility of  the relations between character triples and the wreath product constructions (see Sp\"ath \cite[Theorem 5.2]{art:Spath2017} and \cite[Theorem 2.21]{ch:Spath2018}). Theorem  \ref{th:Derived} below improves \cite[Theorem 5.2.12]{book:Marcus1999} in several ways, by taking into account all the additional structure that we deal with. As already noted by Zimmermann \cite{art:Zimm2002}, a certain ``$p'$-condition" on the order of the grading  groups, which appears in \cite[Theorem 5.2.12]{book:Marcus1999}, is actually not needed in the case of derived equivalences, but is needed in the case of Rickard equivalences. Finally, Theorem \ref{th:block_wreath} and Corollary \ref{cor:wreath-rel-b} are the main results of this paper, and establish the compatibility of the relation $\ge_b$ between module triples with wreath products of derived equivalences.

The material is organized as follows: In Section \ref{sec:Preliminaries} we introduce the general notations and we recall from \cite{art:MM2} our basic definitions of a $G$-graded algebra over a $G$-graded $G$-acted algebra $\C$, of a $G$-graded bimodule over $\C$, and of a $G$-graded Morita equivalence over $\C$, where $G$ is a finite group. We also recall from Harris \cite{art:Harris2005,art:Harris2007} some needed facts on the behavior of the Jacobson radical, of centralizers, and of the Brauer map with respect to tensor products. For general concepts and results we refer to \cite{book:Linck2018} and \cite{book:Marcus1999}.

In Section \ref{sec:DerivedWreath} we recall from \cite{art:MVA2} the construction of wreath product for group graded algebras and bimodules over $\C$, and we extend them to chain complexes of $G$-graded bimodules over $\C$
Our main result of this section, Theorem \ref{th:Derived}, says that the wreath product between a chain complex of $G$-graded bimodules over $\C$ and the symmetric group of order $n$, $S_n$, is a complex of $G\wr S_n$-graded bimodules over $\C^{\otimes n}$, and moreover, if the given complex induces a $G$-graded derived (respectively Rickard) equivalence over $\C$, then its wreath product with $S_n$ (respectively a $p'$-subgroup of $S_n$) will induce a group graded derived (respectively Rickard) equivalence over $\C^{\otimes n}$. Our group graded algebras here are block extensions, but it is clear that most of the statements are true for more general group graded algebras.

In Section \ref{sec:ModuleTriples} we recall from \cite{art:MM2} the definitions of a module triple and that of the relation $\ge_c$ between module triples. In Proposition \ref{prop:wreath_c} we prove that the relation $\geq_c$ is compatible with wreath products of $G$-graded derived equivalences over $\C$.

In Section \ref{sec:ModuleTriplesBlocks} we introduce the relation $\ge_b$ between module triples as a refinement of the relation $\ge_c$, by using the Harris-Kn\"orr correspondence (see Definition \ref{d:rel-b}). Note that our definition does not fully cover \cite[Definition 4.2]{ch:Spath2018}, because there block induction in a more general situation is considered. We also introduce in Definition \ref{def:Brauer_derived} a notion of a derived equivalence compatible with the Brauer map. This a a weaker condition that that of a splendid or basic equivalence, and is inspired by the results of \cite{art:Marcus2020}, which connect basic Morita equivalences with the main result of Dade \cite{art:Dade1973}. We prove in  Proposition \ref{prop:block_order} that the relation $\ge_b$ between module triples is a consequence of a certain group graded derived equivalence compatible with the Brauer map. In Theorem \ref{th:block_wreath} and Corollary \ref{cor:wreath-rel-b} we prove that these equivalences, and the relations $\ge_b$ between module triples induced by them are compatible with wreath products.

\section{Preliminaries} \label{sec:Preliminaries}

\begin{subsec}\label{subsec:Hypo} All rings in this paper are associative with identity $1 \neq 0$ and all modules are left (unless otherwise specified) unital and finitely generated. Throughout this article $n$ will represent an arbitrary nonzero natural number.

We consider {a finite group $G$,} a $p$-modular system $(\K,\O,\k)$, where $\O$ is a complete discrete valuation ring, $\K$ is the field of fractions of $\O$  and $\k=\O/J(\O)$ is its residue field. We assume that $\k$ is algebraically closed, and that $\K$ contains all the $|G|$-th roots of unity.
\end{subsec}

\begin{subsec} Let  $N$ be a normal subgroup of $G$, and denote $\G:=G/N$. Note that most results in this paper will use ``$\G$-gradings'', although this is not everywhere needed. The reason is given by the fact that our main applications concern the  $\G$-graded algebra $A=b\O G$, where $b$ is a $\G$-invariant block of $\O N$.
\end{subsec}

We recall from \cite{art:MM2} the following definitions:

\begin{definition} \label{def:defs-GG-alg}
  An $\mathcal{O}$-algebra $\C$ is a \textit{$\G$-graded $\G$-acted algebra} if
 \begin{enumerate}
	\item[(1)] $\C$ is $\G$-graded, and we write $\C=\bigoplus_{\bar{g}\in \bar G}\C_{\bar{g}}$;
	\item[(2)] $\G$ acts on $\C$ (on the left);
	\item[(3)] for all $\bar{g},\bar{h}\in \G$ and for all $c\in \C_{\bar{h}}$ we have  $\tensor*[_{}^{\bar{g}}]{c}{_{}^{}}\in \C_{\tensor*[_{}^{\bar{g}}]{\bar{h}}{_{}^{}}}$.
 \end{enumerate}
\end{definition}

\begin{definition} \label{def:algebraOverC}
Let $\C$ be a $\G$-graded $\G$-acted algebra. We say the $A$ is a $\G$-graded $\mathcal{O}$-algebra over $\C$   if there is a  $\G$-graded $\G$-acted algebra homomorphism
\[\zeta:\C\to C_A(B),\] where $B:=A_1$ and $C_A(B)$ is the centralizer of $B$ in $A$,  i.e. for any $\bar{h}\in \G$ and $c\in \C_{\bar{h}}$, we have $\zeta(c)\in C_A(B)_{\bar{h}}$, and for every $\bar{g}\in\G$, $\zeta(\tensor*[^{\bar{g}}]{c}{})=\tensor*[^{\bar{g}}]{\zeta(c)}{}$.
\end{definition}

\begin{definition}
 Let $A$ and $A'$ be two $\G$-graded crossed products over a $\G$-graded $\G$-acted algebra $\C$, with structure maps $\zeta$ and $\zeta'$, respectively.

{\rm a)}	We say that $\tilde{M}$ is a $\G$-graded $(A,A')$-bimodule over $\C$ if:
	\begin{enumerate}
		\item[(1)] $\tilde{M}$ is an $(A,A')$-bimodule;
		\item[(2)] $\tilde{M}$ has a decomposition $\tilde{M}=\bigoplus_{\bar{g}\in\G}\tilde{M}_{\bar{g}}$ such that $A_{\bar{g}}\tilde{M}_{\bar{x}}A'_{\bar{h}}\subseteq \tilde{M}_{\bar{g}\bar{x}\bar{h}}$, for all $\bar{g}, \bar{x},\bar{h}\in \G$;
		\item[(3)] $\tilde{m}_{\bar{g}} c=\tensor*[^{\bar{g}}]{c}{} \tilde{m}_{\bar{g}}$, for all $c\in \C$, $\tilde{m}_{\bar{g}}\in\tilde{M}_{\bar{g}}$, $\bar{g}\in \G$, where $c\tilde {m} = \zeta(c)\tilde{m}$ and $\tilde{m}c=\tilde{m}\zeta'(c)$, for all $c\in \C$, $\tilde{m}\in\tilde{M}$.
	\end{enumerate}

{\rm b)} $\G$-graded $(A,A')$-bimodules over $\C$ form a category,
where the morphisms between $\G$-graded $(A,A')$-bimodules over $\C$ are just homomorphisms between $\G$-graded $(A,A')$-bimodules.
\end{definition}

\begin{definition}
Let $A$ and $A'$ be two $\G$-graded crossed products over a $\G$-graded $\G$-acted algebra $\C$, and let $\tilde{M}$ be a $\G$-graded $(A,A')$-bimodule over $\C$. Clearly, the $A$-dual $\tilde{M}^{\ast}=\Hom{A}{\tilde{M}}{A}$ of $\tilde{M}$ is a $\G$-graded $(A',A)$-bimodule over $\C$. We say that  $\tilde{M}$ induces a $\G$-graded Morita equivalence over $\C$ between $A$ and $A'$, if $\tilde{M}\otimes_{A'}\tilde{M}^{\ast}\simeq A$ as $\G$-graded $(A,A)$-bimodules and $\tilde{M}^{\ast}\otimes_{A}\tilde{M}\simeq A'$ as $\G$-graded $(A',A')$-bimodules.
\end{definition}

We will need certain properties of tensor products of algebras endowed with group actions or group gradings. We rely on {the results of} Harris \cite{art:Harris2005,art:Harris2007}, which extend the results of K\"ulshammer \cite{art:Ku1993}, and of  Alghamdi and Khammash \cite{art:AK2002}.

Let $A$ and $A'$ be two $\G$-graded crossed products, hence $A\otimes A'$ is a $\G\times \bar G$-graded crossed product with 1-component $B\otimes B'$. We assume from now on, that $A$  and $A'$ are free and finitely generated as $\O$-modules.

\begin{subsec} \label{subsec:TensorJacobson} We start with the graded Jacobson radical.  {By \cite[Proposition 1.5.11]{book:Marcus1999} and \cite[Section 2]{art:Harris2005}} we have{:}
\begin{align*} J_{\mathrm{gr}} (A\otimes A') &= J(B\otimes B') (A\otimes A') \\
&   = (J(B)\otimes B' +B\otimes J(B')) (A\otimes A')  \\
&   = J(B)A\otimes A' + A\otimes J(B')A'  \\
&   = J_{\mathrm{gr}}(A)\otimes A' + A\otimes J_{\mathrm{gr}}(A').
\end{align*}
From this equality it easily follows  that
\[ A\otimes A' /  J_{\mathrm{gr}} (A\otimes A') \simeq  A/ J_{\mathrm{gr}} (A) \otimes A'/J_{\mathrm{gr}}(A').\]
Moreover, { these results imply the following:}
$$\begin{array}{rcl}
J_{\mathrm{gr}} (A^{\otimes n})  = A^{\otimes n}J (B^{\otimes n}) &\text{and} &
A^{\otimes n} / J_{\mathrm{gr}} (A^{\otimes n}) \simeq (A/J_{\mathrm{gr}} (A) )^{\otimes n},
\end{array}$$
{where $A^{\otimes n}:=A\otimes\ldots\otimes A$ ($n$ times) and $B^{\otimes n}$ is its identity component.}
\end{subsec}

\begin{subsec} \label{subsec:TensorCentralizer} Because under our assumptions the Hom functors behave well with respect to tensor products, we have the isomorphisms
\begin{align*}
C_A(B)\otimes C_{A'}(B') &\simeq \mathrm{End}_{A\otimes B^{\mathrm{op}}}(A)^{\mathrm{op}} \otimes \mathrm{End}_{A'\otimes {B'}^{\mathrm{op}}}(A')^{\mathrm{op}} \\
   &\simeq  \mathrm{End}_{(A\otimes A')\otimes (B\otimes B')^{\mathrm{op}}}(A\otimes A')^{\mathrm{op}}  \\
   &\simeq C_{A\otimes A'}(B\otimes B')
\end{align*}
of $\bar G\times \bar G'$-graded $\bar G\times \bar G'$-acted algebras.
\end{subsec}

\begin{subsec} \label{subsec:Brauer_Construction} Finally, for the Brauer construction, if $A$ is a $G$-acted $\mathcal{O}$-algebra and $Q$ is a $p$-subgroup of $G$, by \cite[Section 1]{art:Harris2007}, we have the commutative diagram
\[\xymatrix{
  (A^Q)^{\otimes n} \ar[rr]  \ar_{\mathrm{Br}_Q^{\otimes n}}[d] &  &  (A^{\otimes n})^{Q^n}  \ar^{\mathrm{Br}_{Q^n}}[d] \\
 A(Q)^{\otimes n}  \ar[rr]  &  &    A^{\otimes n}(Q^n)
}\]
of $N_G(Q)^n$-acted algebras, where the horizontal maps are isomorphisms.
\end{subsec}

\section{Derived equivalences for wreath products} \label{sec:DerivedWreath}

Consider the notations from Section \ref{sec:Preliminaries}.

\begin{subsec}
The wreath product $\G\wr S_n$ is the semidirect product $\G^n\rtimes S_n$, where the symmetric group $S_n$  acts on $\G^n$ (on the left) by permuting the components:
\[\tensor*[^{\sigma}]{{(g_1,\ldots,g_n)}}{}:=(g_{\sigma^{-1}(1)},\ldots,g_{\sigma^{-1}(n)}),\]
{for all $g_1,\ldots,g_n\in \G$ and $\sigma \in S_n$.}
{More exactly, the elements of $\G\wr S_n$ are of the form $((g_1,\ldots,g_n),\sigma)$, and the multiplication is:
$$((g_1,\ldots,g_n),\sigma)((h_1,\ldots,h_n),\tau):=((g_1,\ldots,g_n)\cdot \tensor*[^{\sigma}]{{(h_1,\ldots,h_n)}}{},\sigma\tau),$$
for all $g_1,\ldots,g_n,h_1,\ldots,h_n\in \G$ and $\sigma,\tau\in S_n$.}

Similarly, {if $A$ is an $\O$-algebra,} the wreath product $A\wr S_n$ is the skew group algebra
\[A\wr S_n := A^{\otimes n}\otimes \O S_n\]
between $A^{\otimes n}$ and $S_n$, with multiplication
\[\para{\para{a_1\otimes\ldots\otimes a_n}\otimes\sigma}\para{\para{b_1\otimes\ldots\otimes b_n}\otimes\tau} = (a_1b_{\sigma^{-1}(1)}\otimes\ldots\otimes a_nb_{\sigma^{-1}(n)}) \otimes \sigma\tau,\]
{for all $(a_1\otimes\ldots\otimes a_n)\otimes\sigma$, $(b_1\otimes\ldots\otimes b_n)\otimes\tau\in A\wr S_n$.}
\end{subsec}
	{We recall from \cite{art:MVA2}, Lemma 4.3 under the following form:}
\begin{lemma} \label{lemma:WR_algebra}
Let $A$ be a $\G$-graded crossed product over the $\G$-graded $\G$-acted algebra $\C$. The following statements hold:
\begin{enumerate}
	\item [1)] $\C^{\otimes n}$ is a $\G\wr S_n$-acted $\G^n$-graded algebra, where
\[ \tensor*[^{{((g_1,\ldots,g_n),\sigma)}}]{{(c_1\otimes\ldots\otimes c_n)}}{}:=\tensor*[^{{g_{1}}}]{{c_{\sigma^{-1}(1)}}}{}\otimes\ldots\otimes\tensor*[^{{g_{n}}}]{{c_{\sigma^{-1}(n)}}}{}. \]
	\item [2)]  $A\wr S_n$ is a $\G\wr S_n$-graded crossed product over $\C^{\otimes n}$, with $((g_1,\ldots,g_n),\sigma)$-component
\[(A\wr S_n)_{((g_1,\ldots,g_n),\sigma)}:=((A_{g_1}\otimes\ldots\otimes A_{g_n})\otimes\O\sigma),\]
for each $((g_1,\ldots,g_n),\sigma)\in \bar{G}\wr S_n$, and with structural {$\G\wr S_n$-graded $\G\wr S_n$-acted} algebra homomorphism
\[\zeta_{\mathrm{wr}}:\C^{\otimes n}\to C_{A\wr S_n}(B^{\otimes n})\] given by the composition
\[\zeta^{\otimes n}:\C^{\otimes n}\to C_{A}(B)^{\otimes n}\subseteq C_{A\wr S_n}(B^{\otimes n}).\]
\end{enumerate}
\end{lemma}

\begin{subsec} \label{subsec:S_n_gradings} Let $A$ and $A'$ be two $\G$-graded crossed products over the $\G$-graded $\G$-acted algebra $\C$, with identity components $B$ and $B'$ respectively.

If $\tilde{M}$ is an $(A,A')$-bimodule, the action of $S_n$ on $\tilde{M}^{\otimes n}$ is defined by
\[\tensor[^{\sigma}]{\para{\tilde{m}_1\otimes\ldots\otimes \tilde{m}_n}}{}:=\tilde{m}_{\sigma^{-1}(1)}\otimes\ldots\otimes \tilde{m}_{\sigma^{-1}(n)},\]
{for all $\tilde{m}_1,\ldots, \tilde{m}_n\in \tilde{M}$ and $\sigma\in S_n$.}
As an $\O$-module, the wreath product $\tilde{M}\wr S_n$ is
\[\tilde{M}\wr S_n:=\tilde{M}^{\otimes n}\otimes \O S_n.\]

Note that regarding $A\wr S_n$ and $A'\wr S_n$ as $S_n$-graded algebras, we may consider the diagonal subalgebra:
$$\Delta_{S_n}:=\Delta_{S_n}(A\wr S_n\otimes (A'\wr S_n)^{\mathrm{op}}):=(A\wr S_n\otimes (A'\wr S_n)^{\mathrm{op}})_{\delta(S_n)},$$
where $\delta(S_n):=\set{(\sigma,\sigma^{-1})\mid \sigma\in S_n}$. It is easy to see that:
$$\Delta_{S_n}(A\wr S_n\otimes (A'\wr S_n)^{\mathrm{op}})\simeq (A\otimes A'^{\mathrm{op}})\wr S_n,$$
as $\G\wr S_n$-graded algebras, and thus we have that
$$\tilde{M}\wr S_n\simeq (A\wr S_n\otimes (A'\wr S_n)^{\mathrm{op}})\otimes_{\Delta_{S_n}}\tilde{M}^{\otimes n},$$
as $\G\wr S_n$-graded $(A\wr S_n,A'\wr S_n)$-bimodules.
\end{subsec}

We recall  \cite[Theorem 5.3]{art:MVA2}, which extends \cite[Theorem 5.1.21]{book:Marcus1999} to the case of group graded Morita equivalences over a group graded group acted algebra:

\begin{theorem} \label{th:main_of_art_4}
Let $\tilde{M}$ be a $\G$-graded $(A,A')$-bimodule over $\C$, with identity component $M$. Then, the following {statements} hold:
	\begin{enumerate}
		\item [(1)] $\tilde{M}\wr S_n$ is a $\G\wr S_n$-graded $\para{A\wr S_n,A'\wr S_n}$-bimodule over $\C^{\otimes n}$, with scalar multiplication
\begin{align*}
((a_1\otimes\ldots\otimes a_n)\otimes\sigma)&\para{\para{\tilde{m}_1\otimes\ldots\otimes \tilde{m}_n}\otimes\tau}  ((a'_1\otimes\ldots\otimes a'_n)\otimes\pi)    \\
 &= (a_1\otimes\ldots\otimes a_n)\cdot {}^\sigma (\tilde{m}_1\otimes\ldots\otimes \tilde{m}_n)  \cdot {}^{\sigma\tau}(a'_1\otimes\ldots\otimes a'_n) \otimes \sigma\tau\pi,
\end{align*} and with $((g_1,\ldots,g_n),\sigma)$-component
\[(\tilde{M}\wr S_n)_{((g_1,\ldots,g_n),\sigma)}=(\tilde{M}_{g_1}\otimes\ldots\otimes\tilde{M}_{g_n}) \otimes\O\sigma.\]
		\item [(2)] There are isomorphisms of $\G\wr S_n$-graded $\para{A\wr S_n,A'\wr S_n}$-bimodules over $\C^{\otimes n}$:
	\[
	\begin{array}{l}
		f:\para{A\wr S_n}\otimes_{B^{\otimes n}}M^{\otimes n}\to \tilde{M}\wr S_n,\\
		((a_1\otimes\ldots\otimes a_n)\otimes\sigma)\otimes (m_1\otimes\ldots\otimes m_n)\mapsto ((a_1\otimes\ldots\otimes a_n)\cdot {}^\sigma(m_1\otimes\ldots\otimes m_n))\otimes\sigma,
	\end{array}
	\]		
 and
	\[
	\begin{array}{l}
		g:M^{\otimes n}\otimes_{B'^{\otimes n}}\para{A'\wr S_n}\to \tilde{M}\wr S_n,\\
		(m_1\otimes\ldots\otimes m_n)\otimes((a'_1\otimes\ldots\otimes a'_n)\otimes\sigma) \mapsto ((m_1\otimes\ldots\otimes m_n)\cdot (a'_1\otimes\ldots\otimes a'_n))\otimes\sigma. \\
	\end{array}
	\]
		\item [(3)] If $\tilde{M}$ induces a $\G$-graded Morita equivalence over $\C$ between $A$ and $A'$, then $\tilde{M}\wr S_n$ induces a $\G\wr S_n$-graded Morita equivalence over $\C^{\otimes n}$ between $A\wr S_n$ and $A'\wr S_n$.
	\end{enumerate}
\end{theorem}

\begin{subsec}  Now, if $\tilde{X}$ is a chain complex of {$\G$-graded} $(A,A')$-bimodules {over $\C$} which induces a $\bar G$-graded derived or Rickard equivalence  between $A$ and $A'$, we want to extend the  results of   \cite[Section 5.1.C]{book:Marcus1999},  to obtain  a $\G\wr S_n$-graded derived or Rickard equivalence over $\C^{\otimes n}$ between $A\wr S_n$ and $A'\wr S_n$. In the case of Rickard equivalences, some additional condition will be needed.

{Note that by a derived equivalence we mean an equivalence between the bounded derived categories $\Db{A}$ and $\Db{A'}$ induced by a two-sided tilting complex as in \cite[Section 6.2]{book:KZ1998}, while by a Rickard equivalence, we mean an equivalence between the {bounded} chain homotopy categories $\Hb{A}$ and $\Hb{A'}$ induced by a split endomorphism tilting complex, as presented by Rickard in \cite[Section 9.2.2]{book:KZ1998}; in this case it is essential that $A$ and $A'$ are symmetric algebras.}
\end{subsec}

\begin{subsec} \label{subsec:wreath_complexes} Recall (see, for instance \cite[Section 4.1]{book:Benson1991}) that $S_n$ acts on {$\tilde{X}^{\otimes n}:=\tilde{X}\otimes\ldots\otimes \tilde{X}$ ($n$ times)}. By {\cite[Lemma 5.2.11]{book:Marcus1999}}, this action can be defined as follows\textcolor{red}{:} Denote $C_2=\{\pm1\}$, and observe that $S_n$ acts on the abelian group $\mathrm{Fun}(C_2^n,C_2)$ of functions from $C_2^n$ to $C_2$; for $i\in\mathbb{Z}$ denote also $\hat i=(-1)^i$. Then there is a $1$-cocycle $\epsilon\in Z^1(S_n, \mathrm{Fun}(C_2^n,C_2))$ such that
\[{}^\sigma(x_{i_1}\otimes\dots \otimes x_{i_n})=\epsilon_\sigma(\hat i_1,\dots, \hat i_n)x_{i_{\sigma^{-1}(1)}}\otimes\dots \otimes x_{i_{\sigma^{-1}(n)}},\]
where $x_{i_j}$ belongs to the $j$-th factor of $\tilde X^{\otimes n}$, and has degree $i_j\in\mathbb{Z}$. In our situation, $\tilde X^{\otimes n}$ is a complex of $\bar G^n$-graded $(A^{\otimes n}, {A'}^{\otimes n})$-bimodules  over $\C^{\otimes n}$, and even more, a complex of $\bar G^n$-graded $(A\otimes {A'}^{\mathrm{op}})\wr S_n$-modules.

We may therefore consider the wreath product
\[\tilde X\wr S_n=\tilde X^{\otimes n}\otimes \mathcal{O}S_n .\]
\end{subsec}

\begin{theorem} \label{th:Derived} Let $\tilde{X}$ be a complex of $\G$-graded $(A,A')$-bimodules over $\C$, with identity component $X$. Then, the following statements hold:
\begin{enumerate}
		\item [1)] $\tilde{X}\wr S_n$ is a complex of $\G\wr S_n$-graded $\para{A\wr S_n,A'\wr S_n}$-bimodules over $\C^{\otimes n}$, isomorphic to $\para{A\wr S_n}\otimes_{B^{\otimes n}}X^{\otimes n}$ and to $X^{\otimes n}\otimes_{B'^{\otimes n}}\para{A'\wr S_n}$.
		\item [2)]  If $\tilde{X}$ induces a $\G$-graded derived equivalence  between $A$ and $A'$, then {$\tilde{X}\wr S_n$} induces a $\G\wr S_n$-graded derived equivalence over $\C^{\otimes n}$ between $A\wr S_n$ and $A'\wr S_n$.
		\item [3)]  If $\tilde{X}$ induces a $\G$-graded Rickard equivalence  between $A$ and $A'$, and if $\Sigma$ is a $p'$-subgroup of $S_n$, then $\tilde{X}\wr \Sigma$ induces a $\G\wr\Sigma$-graded Rickard equivalence over $\C^{\otimes n}$ between $A\wr \Sigma$ and $A'\wr \Sigma$.
	\end{enumerate}
\end{theorem}

\begin{proof} 1) We use the fact that the constructions presented in \ref{subsec:S_n_gradings}, Theorem \ref{th:main_of_art_4} and \ref{subsec:wreath_complexes} are functorial. More precisely, by using {\cite[Lemma 1.6.3]{book:Marcus1999} and }\cite[Proposition 2.11]{art:MM2} we deduce that $(A\wr S_n\otimes (A'\wr S_n)^{\op})\otimes_{\Delta_{S_n}}-$, $A\wr S_n\otimes_{B^{\otimes n}}-$ and $-\otimes_{B'^{\otimes n}}A'\wr S_n$ are naturally isomorphic equivalences of categories, from the category of complexes of $(A\otimes A'^{\op})\wr S_n$-modules {over $\C^{\otimes n}$} to the category of complexes of $\G\wr S_n$ -graded $(A\wr S_n,A'\wr S_n)$-bimodules over $\C^{\otimes n}$.

2) Let $\tilde{Y}=\RHom{A}{\tilde{X}}{A}$ be the $A$-dual of $\tilde{X}$, hence $\tilde{Y}$ is a complex of $\G$-graded $(A',A)$-bimodules over $\C$, by \cite[Proposition 2.12.(2)]{art:MM2}.
By assumption, the canonical map
$$\tilde{X}\lotimes_{A'}\tilde{Y}\to A$$
 of complexes of $\G$-graded $(A,A)$-bimodules is an isomorphism in the derived category $\Db{A\otimes A^{\op}}$, that is, it induces an isomorphism between the homology groups of the above complexes. Consequently, we get a map
$$f:\tilde{X}^{\otimes n}\lotimes_{A'^{\otimes n}}\tilde{Y}^{\otimes n}\to A^{\otimes n}$$
of complexes of $\G^n$-graded $(A^{\otimes n},A^{\otimes n})$-bimodules over $\C^{\otimes n}$, which is a quasi-isomorphism.
Now, by \ref{subsec:wreath_complexes}, $\tilde{X}^{\otimes n}$ extends to a complex of $\Delta_{S_n}\simeq (A\otimes A'^{\mathrm{op}})\wr S_n$-modules. It follows, as in the proof of \cite[Theorem 5.2.5]{book:Marcus1999}, that $\tilde{Y}^{\otimes n}$ extends to a complex of $(A'\otimes A^{\mathrm{op}})\wr S_n$-modules, and that the canonical map $f$ is $(A\otimes A^{\mathrm{op}})\wr S_n$-linear.
Observe that $f$ still induces an isomorphism between homology groups, hence $f$ is an isomorphism in the derived category  $\Db{(A\otimes A^{\op})\wr S_n}$, and it also preserves $\G^n$-gradings.
By \cite[Proposition 2.11]{art:MM2} we deduce that $f$ induces an isomorphism
$$f_{\mathrm{wr}}:\tilde{X}\wr S_n\otimes \tilde{Y}\wr S_n\to A\wr S_n,$$
in the bounded derived category of $\G\wr S_n$-graded $(A\wr S_n,A\wr S_n)$-bimodules. This argument shows that $\tilde{X}\wr S_n$ induces a $\G\wr S_n$-graded derived equivalence over $\C^{\otimes n}$ between $A\wr S_n$ and $A'\wr S_n$.

3) Let $\tilde{Y}$ be the $\O$-dual of $\tilde{X}$, hence $\tilde{Y}$ is a complex of $\G$-graded $(A',A)$-bimodules over $\C$. By assumptions, there are $\G$-grade preserving canonical isomorphisms:
\[h:\tilde{X}\otimes_{A'}\tilde{Y}\to A\qquad \text{ and }\qquad h':A\to \tilde{X}\otimes_{A'}\tilde{Y},\]
in the homotopy category $\Hb{A\otimes A^{\op}}$, inverse of each other. Then,
\[h^{\otimes n}:\tilde{X}^{\otimes n}\otimes_{A'^{\otimes n}}\tilde{Y}^{\otimes n}\to A^{\otimes n}\qquad \text{ and }\qquad h'^{\otimes n}:A^{\otimes n}\to \tilde{X}^{\otimes n}\otimes_{A'^{\otimes n}}\tilde{Y}^{\otimes n},\]
are isomorphisms in the homotopy category of $\G^n$-graded $(A^{\otimes n},A^{\otimes n})$-bimodules. As above, it follows by \ref{subsec:wreath_complexes} that $h^{\otimes n}$ and $h'^{\otimes n}$ are in fact $(A\otimes A^{\op})\wr S_n$-linear. Since $\Sigma$ is a $p'$-subgroup of $S_n$, the final part of the proof of \cite[Theorem 3.4.(b)]{art:Marcus1996} shows that we get the isomorphisms:
\[h\wr \Sigma:\tilde{X}\wr \Sigma\otimes_{A'\wr \Sigma}\tilde{Y}\wr \Sigma\to A\wr \Sigma\qquad \text{ and }\qquad h'\wr \Sigma:A\wr \Sigma\to \tilde{X}\wr \Sigma\otimes_{A'\wr \Sigma}\tilde{Y}\wr \Sigma,\]
in the homotopy category of $\G\wr \Sigma$-graded $(A\wr \Sigma,A\wr \Sigma)$-bimodules. By symmetry, the statement is proved.
\end{proof}

\section{Module triples} \label{sec:ModuleTriples}

\begin{subsec}\label{subsec:ModuleTriples_Relations}  Additionally to the assumptions from \ref{subsec:Hypo}, we will consider $G'$ to be a subgroup of $G$ such that  $G=G'N$ and $C_G(N)\subseteq G'$, and let $N'=G'\cap N$, hence  $\G=G/N\simeq G'/N'$.

Let $b\in Z(\O N)$ and ${b'}\in Z(\O N')$ be two $\G$-invariant block idempotents. We denote
\[ A:=b\O G, \qquad A':={b'}\O G',\qquad B:=b\O N, \qquad B': = {b'}\O N', \]
hence $A$ and $A'$ are  $\G$-graded crossed products, with  1-compo\-nents $B$ and $B'$ respectively. We also have that $A$ and $A'$ are  $\G$-graded algebras over a $\G$-graded $\G$-acted $\O$-algebra $\C=\mathcal{O}C_G(N)$,  with structural maps $\zeta:\C\to C_A(B)$ and $\zeta':\C\to C_{A'}(B')$, as in Definition \ref{def:algebraOverC}, given by inclusion. We denote $$\K B=\K\otimes_{\O}B=(1\otimes b)\K N, \qquad \K B'=\K\otimes_{\O}B'=(1\otimes b')\K N'.$$

Let $V$ be a $G$-invariant simple $\K B$-module and $V'$ be a $G'$-invariant simple $\K B'$-module. In this situation, we say that $(A,B,V)$ is a module triple, and we will consider its endomorphism algebra
\[E(V):=\text{End}_{\K A}\left(\K A\otimes_{\K B}V\right)^{\op}.\]
\end{subsec}

We recall from \cite{art:MM2}, the  relation $\geq_c$ between module triples.

\begin{definition}
Let $\triple{A}{B}{V}$ and $\triple{A'}{B'}{V'}$ be two module triples.  We write $(A,B,V)\geq_c (A',B',V')$ if there exists a $\G$-graded algebra isomorphism
\[E(V)=\mathrm{End}_{\K A}(\K A\otimes_{\K B}V)^{\op}\to E(V')=\mathrm{End}_{\K A'}(\K A'\otimes_{\K B'}V')^{\op}\] such that the diagram
	\[\xymatrix@C+=3cm{
		E(V) \ar@{->}[r]^{\sim} & E(V')\\
	    \K \C  \ar@{->}[u] \ar@{=}[r]^{\mathrm{id}_{\K \C}} & \K \C, \ar@{->}[u]
}\]
of $\G$-graded $\K$-algebras is commutative,	where $\K \C=\K C_G(N)$ is regarded as a $\G$-graded $\G$-acted $\K$-algebra, with 1-component $\K Z(N)$.
\end{definition}

The next result is motivated by  \cite[Theorem 2.21]{ch:Spath2018}.

\begin{proposition} \label{prop:wreath_c} Consider the module triples $(A,B,V)$ and $(A',B',V')$. If $A$ and $A'$ are $\G$-graded derived equivalent over $\C$ such that $V$ corresponds to $V'$, then
\[(A\wr S_n,B^{\otimes n},V^{\otimes n})\geq_c(A'\wr S_n,B'^{\otimes n},V'^{\otimes n}).\]
\end{proposition}

\begin{proof} Observe that $(A\wr S_n,B^{\otimes n},V^{\otimes n})\text{ and }(A'\wr S_n,B'^{\otimes n},V'^{\otimes n})$ are module triples. Indeed, $N^n$ is a normal subgroup of $G\wr S_n$, $G'\wr S_n$ is a subgroup of $G\wr S_n$,  $N'^n=(G'\wr S_n)\cap (N^n)$, and $G\wr S_n=(G'\wr S_n) (N^n)$. Moreover, it is easy to see that
\[G\wr S_n/N^n\simeq G'\wr S_n/N'^n\simeq \G\wr S_n.\]
It is also clear that $b^{\otimes n}$ and $b'^{\otimes n}$ are $\G\wr S_n$-invariant block idempotents in $Z(\O N^n)$ and $Z(\O N'^n)$ respectively, and moreover, $A\wr S_n\simeq b^{\otimes n}\O(G\wr S_n)$ and $A'\wr S_n \simeq b'^{\otimes n}\O(G'\wr S_n)$. By Lemma \ref{lemma:WR_algebra} we have that $A\wr S_n$ and $A'\wr S_n$ are strongly $\G\wr S_n$-graded algebras over $\C^{\otimes n}$ with identity components $B^{\otimes n}=b^{\otimes n}\mathcal{O}N^n$ and $B'^{\otimes n}=b'^{\otimes n}\mathcal{O}N'^n$ respectively. Finally, it is straightforward that $V^{\otimes n}$ is a $G\wr S_n$-invariant simple $\K B^{\otimes n}$-module, that $V'^{\otimes n}$ is a $G'\wr S_n$-invariant simple $\K B'^{\otimes n}$-module, and that $V'^{\otimes n}$ corresponds to $V^{\otimes n}$.

Now, by Theorem \ref{th:Derived}, $A\wr S_n$ and $A'\wr S_n$ are $\G\wr S_n$-graded derived equivalent over $\C^{\otimes n}$, hence the conclusion follows by \cite[Theorem 6.7]{art:MM2}.
\end{proof}

\section{Module triples and blocks} \label{sec:ModuleTriplesBlocks}

We keep the notations of the preceding section.

Sp\"ath also considered in \cite{art:Spath2017}, \cite{ch:Spath2017} and \cite{ch:Spath2018} the relation $\ge_b$ between character triples. This relation is a refinement of $\ge_c$, and involves block induction, see \cite[Definition 4.2]{ch:Spath2018}. We show in this section that certain group graded derived equivalences compatible with the Brauer map imply the relation $\ge_b$ between the corresponding triples, and are also compatible with wreath products.

We are going to use the Brauer map and basic equivalences between blocks, introduced by L.~Puig in \cite{book:Puig1997}. Then \cite[Remark 4.3 (c)]{ch:Spath2018} leads us to the following setting.

\begin{definition} \label{d:rel-b} We assume that the block $b$ has defect group $Q$, $G'=N_G(Q)$, $N'=N_N(Q)$, and $b'$ is the Brauer correspondent of $b$. Let $(A,B,V)$ and $(A',B',V')$ be two module triples. We write \[(A,B,V)\geq_b (A',B',V')\] if the following conditions are satisfied:
\begin{enumerate}
\item $(A,B,V)\geq_c (A',B',V')$;
\item For any subgroup $N\le J\le G$, if the simple $\mathcal{O}J$-module $W$ covering $V$ corresponds (via condition {\rm(1)}) to the simple $\mathcal{O}J'$-module $W'$ covering $V'$ (where $J'=G'\cap J$), then the block $\beta$ of $\mathcal{O}J$ to which $W$ belongs is the Harris-Kn\"orr correspondent of the  block $\beta'$ of $\mathcal{O}J'$ to which $W'$ belongs.
\end{enumerate}
\end{definition}

\begin{subsec} Recall that the Harris-Kn\"orr correspondence \cite{art:HK1985} is a bijection between the blocks of $A$ with defect group $D$ (where $Q\le D$) and the blocks of $A'$ with defect group $D$. This bijection in induced by  the Brauer map \[\mathrm{Br}_Q:A^Q\to A(Q)\] (see \cite[Lemma 3.4]{art:Marcus2020} for an alternative proof).
\end{subsec}

\begin{subsec} \label{subsec:G[b]} Denote
\[\bar C =\bar C_A(B)=C_A(B)/J_{\mathrm{gr}}(C_A(B)).\]
we know from \cite[2.9]{art:Dade1973} that $\bar C$ is a $\bar G[b]$-graded crossed product, where
\[\bar G[b]=\{\bar g\in\bar G   \mid A_{\bar g}\simeq B \textrm{ as } (B,B) \textrm{-bimodules}\}=\{ \bar g\in\bar G  \mid  A_{\bar g}A_{\bar g^{-1}}=B\}.\]
Denote also $\bar C' =\bar C_{A'}(B')=C_{A'}(B')/J_{\mathrm{gr}}(C_{A'}(B'))$.

The main result of Dade \cite{art:Dade1973} says that the Brauer map $\mathrm{Br}_Q$ induces an isomorphism $\bar C\simeq \bar C'$ of $\bar G[b]$-graded $\bar G$-acted algebras. Moreover, by \cite[Theorem 3.7]{art:Marcus2020}, this isomorphism induces the same Harris-Kn\"orr correspondence between the blocks of $A$ and the blocks of $A'$.
\end{subsec}

\begin{subsec} \label{subsec:G[b]-2} Recall also from \cite[Corollary 5.2.6]{book:Marcus1999} that a $\bar G$-graded derived equivalence between $A$ and $A'$ induces yet another isomorphism $\bar C\simeq \bar C'$ of $\bar G[b]$-graded $\bar G$-acted algebras.
\end{subsec}

\begin{definition} \label{def:Brauer_derived} We say that a $\bar G$-graded  derived equivalence between $A$ and $A'$ is {\it compatible with the Brauer map} if the induced isomorphism $\bar C\simeq \bar C'$ of $\bar G[b]$-graded $\bar G$-algebras  from \ref{subsec:G[b]-2} coincides with the isomorphism induced by the Brauer map $\mathrm{Br}_Q$  from \ref{subsec:G[b]}.
\end{definition}

\begin{proposition} \label{prop:block_order} Assume that the complex $\tilde X$ induces a $\bar G$-graded derived equivalence between $A$ and $A'$ compatible with the Brauer map $\mathrm{Br}_Q$, such that the simple $\mathcal{K}B$-module $V$ corresponds to the simple $\mathcal{K}B'$-module $V'$. Then $(A,B,V)\ge_b (A',B',V')$.
\end{proposition}

\begin{proof} The proof of \cite[Theorem 6.7]{art:MM2} works for derived equivalences as well, and shows that $(A,B,V)\ge_c (A',B',V')$. Condition (2) of Definition \ref{d:rel-b} follows from the fact (see \cite[Corollary 5.2.6]{book:Marcus1999}) that the truncated complex $\tilde X_{\bar J}=\oplus_{\bar g\in \bar J}\tilde X_{\bar g}$ induces a $\bar J$-graded derived equivalence between $A_{\bar J}$ and $A'_{\bar J'}$, which is still compatible with the Brauer map $\mathrm{Br}_Q$.
\end{proof}

\begin{remark} By \cite[Corollary 4.4]{art:Marcus2020}, a $\bar G$-graded basic Morita equivalence between $A$ and $A'$  is compatible with the Brauer map $\mathrm{Br}_Q$ in the sense of the above definition.

{Note also that a direct product of Dade $Q$-algebras is also a Dade $Q$-algebra. It follows easily that we get a $\G^n$-graded basic Morita equivalence between $A^{\otimes n}$ and $A'^{\otimes n}$.}

{This in turn,} induces a Morita equivalence between $A\wr S_n\simeq b^{\otimes n}\mathcal{O}(G\wr S_n)$ and $A'\wr S_n\simeq {b'}^{\otimes n}\mathcal{O}(G'\wr S_n)$, by {Theorem \ref{th:main_of_art_4} or \cite[Theorem 5.1.21]{book:Marcus1999}}. However, this equivalence need not be basic (see also \cite[Remark 3.4]{art:Zimm2002}), so we cannot apply the results of \cite{art:Marcus2020} {to deduce} its compatibility with the Brauer map $\mathrm{Br}_{Q^n}$.

Nevertheless, we still have the following result.
\end{remark}

\begin{theorem} \label{th:block_wreath} Assume that the complex $\tilde X$ induces a $\bar G$-graded derived equivalence between $A$ and $A'$ compatible with the Brauer map $\mathrm{Br}_Q$. Then the $\bar G\wr S_n$-graded derived equivalence between $A\wr S_n$ and $A'\wr S_n$ induced by $\tilde X\wr S_n$ is compatible with the Brauer map $\mathrm{Br}_{Q^n}$.
\end{theorem}

\begin{proof} We have that $A\wr S_n\simeq \mathcal{O}(G\wr S_n)b^{\otimes n}$ is a $G\wr S_n/N^n\simeq \bar G \wr S_n$-graded algebra with $1$-component $B^{\otimes n}$. Consider the group $(\bar G\wr S_n)[b^{\otimes n}]$ defined in \ref{subsec:G[b]}, hence $\bar C_{A\wr S_n}(B^{\otimes n})$ is a $(\bar G\wr S_n)[b^{\otimes n}]$-graded crossed product.

By applying the Brauer construction, we get
\[(\mathcal{O}(G\wr S_n))(Q^n)\simeq kC_{G\wr S_n}(Q^n)=kC_{G^n}(Q^n)\simeq (kC_G(Q))^{\otimes n},\]
and in fact, $(A\wr S_n)(Q^n)\simeq A(Q)^{\otimes n}$. This, together with Dade's result presented in \ref{subsec:G[b]} show that $(\bar G\wr S_n)[b^{\otimes n}]\subseteq \bar G^n$, and therefore
\[\bar C_{A\wr S_n}(B^{\otimes n})\simeq \bar C_{A^{\otimes n}}(B^{\otimes n}).\]
Consequently, it is enough to show that the $\bar G^n$-graded derived equivalence between $A^{\otimes n}$ and ${A'}^{\otimes n}$ {induced} by $\tilde X^{\otimes n}$ is compatible with the Brauer map $\mathrm{Br}_{Q^n}$.

We use that all our constructions behave well with respect to tensor products. First, {as in \ref{subsec:TensorCentralizer},} we have (see also \cite[Proposition 1.4.]{art:Harris2007})
\[C_{A^{\otimes n}}(B^{\otimes n})\simeq \mathrm{End}_{A^{\otimes n}\otimes {B^{\otimes n}}^{\mathrm{op}}}(A^{\otimes n})^{\mathrm{op}} \simeq (\mathrm{End}_{A\otimes B^{\mathrm{op}}}(A)^{\mathrm{op}})^{\otimes n} \simeq C_A(B)^{\otimes n}.\]
Next, by {using \ref{subsec:TensorJacobson},} we have
\begin{align*} \bar C^{\otimes n} & =\bar C_A(B)^{\otimes n} = (C_A(B)/J_{\mathrm{gr}(C_A(B)} )^{\otimes n} \\
  &\simeq C_A(B)^{\otimes n}/ J_{\mathrm{gr}}(C_A(B)^{\otimes n}) \\
  & \simeq C_{A^{\otimes n}}(B^{\otimes n}) / J_{\mathrm{gr}}(C_{A^{\otimes n}}(B^{\otimes n})) = \bar C_{A^{\otimes n}}(B^{\otimes n}).
\end{align*}

Let $\varphi_{\tilde{X}}:\bar{C}_A(B)\to \bar{C}_{A'}(B')$ denote the isomorphism induced by $\tilde{X}$, as in {\cite[Subsection 5.3]{art:MM2}}. {Henceforth}, we get the following commutative diagram of isomorphisms of $\G[b]$-graded $\G$-acted algebras:
\[\xymatrix{
  \bar{C}_A(B)^{\otimes n} \ar[rr]^{\varphi_{\tilde{X}}^{\otimes n}}  \ar_{\simeq}[d] &  &  \bar{C}_{A'}(B')^{\otimes n}  \ar^{\simeq}[d] \\
 \bar{C}_{A^{\otimes n}}(B^{\otimes n})  \ar[rr]_{\varphi_{\tilde{X}^{\otimes n}}}  &  &    \bar{C}_{A'^{\otimes n}}(B'^{\otimes n}).
}\]
By \ref{subsec:Brauer_Construction}, we also have the following commutative diagram of isomorphisms of $\G[b]$-graded $\G$-acted algebras:
\[\xymatrix{
  \bar{C}_A(B)^{\otimes n} \ar[rr]^{\mathrm{Br}_Q^{\otimes n}}  \ar_{\simeq}[d] &  &  \bar{C}_{A'}(B')^{\otimes n}  \ar^{\simeq}[d] \\
 \bar{C}_{A^{\otimes n}}(B^{\otimes n})  \ar[rr]_{\mathrm{Br}_{Q^ n}}  &  &    \bar{C}_{A'^{\otimes n}}(B'^{\otimes n}).
}\]
By our assumptions, the isomorphism $\varphi_{\tilde{X}}$ coincides with the isomorphism given by $\mathrm{Br}_{Q}$. The above two commutative diagrams imply that the isomorphisms $\varphi_{\tilde{X}^{\otimes n}}$ and $\mathrm{Br}_{Q^ n}$ also coincide, hence the equivalence induced by $\tilde{X}^{\otimes n}$ is indeed compatible with the Brauer map.
\end{proof}

From Proposition \ref{prop:wreath_c} and Theorem \ref{th:block_wreath} we immediately deduce:

\begin{corollary} \label{cor:wreath-rel-b} Assume that the complex $\tilde X$ induces a $\bar G$-graded derived equivalence over $\mathcal{C}$ between $A$ and $A'$, and that this equivalence is compatible with the Brauer map $\mathrm{Br}_Q$.
Assume also that the simple $\mathcal{K}B$-module $V$ corresponds to the simple $\mathcal{K}B'$-module $V'$. Then
\[(A\wr S_n,B^{\otimes n},V^{\otimes n})\geq_b(A'\wr S_n,B'^{\otimes n},V'^{\otimes n}).\]
\end{corollary}

\begin{remark} We are interested in the relation $\ge_b$ when induced by derived equivalences. However, it is not difficult to show directly, with the methods already used here, that similarly to \cite[Theorem 5.2]{art:Spath2017}, if $(A,B,V)\geq_b (A',B',V')$, then $(A\wr S_n,B^{\otimes n},V^{\otimes n})\geq_b(A'\wr S_n,B'^{\otimes n},V'^{\otimes n})$
\end{remark}

\begin{remark} The following situation is considered in \cite{art:Marcus1996}. Assume that $p\nmid |\bar G|$, $b$ is the principal block of $\mathcal{O}N$, and that $\tilde X$ induces a $\bar G$-graded derived equivalence between $A$ and $A'=\mathcal{O}N_G(Q)b'$, where $b'$ is the principal block of $\mathcal{O}N_N(Q)$ (and of $\mathcal{O}C_N(Q)$). By \cite[Corollary 3.9]{art:Marcus1996}, $\tilde X(Q)$ induces a $C_G(Q)/C_N(Q)$-graded derived autoequivalence of $kC_G(Q)b'$ (which actually lifts to $\mathcal{O}$). It is not difficult to see that this equivalence extends to an $N_G(Q)/C_N(Q)$-graded derived autoequivalence of $A'=\mathcal{O}N_G(Q)b'$. Moreover, the arguments of \cite[Theorem 4.3 and Corollary 4.4]{art:Marcus2020} show that the isomorphism $\bar C\simeq \bar C'$ {induced} by this equivalence coincides with the isomorphism induced by $\mathrm{Br}_Q$.

In order to deal with  arbitrary blocks, one needs to extend the results of \cite{art:Marcus2020} to the case of basic Rickard equivalences. We intend to consider this problem in a subsequent paper.
\end{remark}

\phantomsection

\end{document}